\documentclass[12pt]{amsart}
\usepackage{amsfonts,amsmath,amssymb,amscd,mathrsfs}
\usepackage{xcolor}
\usepackage{verbatim}
\usepackage{enumitem}
\usepackage{hyperref}
\usepackage{tikz-cd}
\newtheorem{theorem}{Theorem}[section]
\newtheorem{lemma}[theorem]{Lemma}
\newtheorem{corollary}[theorem]{Corollary}
\newtheorem{proposition}[theorem]{Proposition}

\theoremstyle{definition}
\newtheorem{definition}[theorem]{Definition}

\newtheorem{remark}[theorem]{Remark}

\theoremstyle{remark}

\def \D{\mathbb{D}}

\def \Z{\mathbb{Z}}

\def \C{\mathbb{C}}
\def \T{\mathbb{T}}
\newcommand{\wt}{\widetilde}
\newcommand{\clb}{\mathcal{B}}
\newcommand{\clh}{\mathcal{H}}
\newcommand{\clk}{\mathcal{K}}
\newcommand{\clr}{\mathcal{R}}
\newcommand{\cli}{\mathcal{I}}
\newcommand{\clt}{\mathcal{T}}
\newcommand{\cls}{\mathcal{S}}
\newcommand{\vp}{\varphi}

\newcommand{\ip}[2]{\left\langle #1,#2\right\rangle}

\newcommand\restr[2]{\ensuremath{\left.#1\right|_{#2}}}

\numberwithin{equation}{section}

\subjclass[2020]{Primary 47B35; Secondary 47A15, 47B47, 46E20}

\keywords{Paired operators, Toeplitz operators, Hankel operators, model spaces, reproducing kernels, truncated Toeplitz operators}

\begin{document}
	\title[Algebraic characterizations of paired operators on model spaces]{Algebraic characterizations of paired operators on model spaces}
	\author{Sudip Ranjan Bhuia}
	\address{Department of Mathematics, Shiv Nadar University, NH91, Tehsil Dadri, Greater Noida, Gautam Buddha Nagar, Uttar Pradesh-201314, India}
	\email{sudipranjanb@gmail.com; sudip.bhuia@snu.edu.in}
	\author{Satyabrata Majee}
	\address{School of Mathematical and Statistical Sciences, Research Focus: Pure and Applied Analytics,  North-West University, Potchefstroom, 2520, South Africa}
	\email{majeesatyabrata@gmail.com, 56337817@mynwu.ac.za}
	\date{\today}

	\begin{abstract}
	 We study paired multiplication operators associated with orthogonal decompositions of $L^2(\mathbb{T})$ arising from Hardy spaces and Sz.-Nagy--Foias model spaces. For a non-constant inner function $\theta$, we characterize $\mathcal{K}_\theta$-paired and triply paired operators by finite-rank commutator identities with the bilateral shift $M_z$. We also apply these ideas to sums of truncated Toeplitz and truncated Hankel operators on $\mathcal{K}_\theta$. Using the displacement characterizations of Sarason and Gu--Ma, we derive a second-order finite-rank displacement formula for such sums and describe the ambiguity of the decomposition in terms of $\mathcal{I}_\theta =\mathcal{T}_\theta \cap \mathcal{H}_\theta$. In the finite-dimensional case $\theta(z)=z^n$, the framework recovers the classical Toeplitz-plus-Hankel matrix recurrence of Bevilacqua--Bonanni--Bozzo.
	\end{abstract}

	\maketitle
	\tableofcontents

	\section{Introduction}
	
	Toeplitz and Hankel operators form two fundamental classes in operator theory, with connections to function theory, interpolation, prediction theory, and systems theory; see, for example, \cite{Peller:Book}. On the Hardy space $H^2(\T)$, the algebraic structure of the Toeplitz operators is governed by the Brown--Halmos theorem, which characterizes Toeplitz operators through their interaction with the unilateral shift. Hankel operators satisfy analogous shift relations involving the backward shift. The interaction between these two classes, and in particular the structure of Toeplitz-plus-Hankel operators, has also been studied extensively. At the matrix level, Bevilacqua, Bonanni and Bozzo~\cite{BBB} showed that an $n\times n$ Toeplitz-plus-Hankel matrix is characterized by the recurrence
	\[
	T_{i-1,j}+T_{i+1,j}=T_{i,j-1}+T_{i,j+1},
	\qquad 1\le i,j\le n-2.
	\]
	
	In parallel, the compression of Toeplitz-type operators to model spaces has led to a rich operator-theoretic framework. Let $\theta$ be a non-constant inner function and let
	\[
	\clk_\theta=H^2\ominus \theta H^2
	\]
	be the corresponding Sz.-Nagy--Foias model space. Sarason initiated the systematic study of truncated Toeplitz operators on $\clk_\theta$ in \cite{Sarason2007}, these operators turn out to be canonical examples of complex symmetric operators in the sense of Garcia and Putinar \cite{Garcia:Putinar}. One of the central features of this theory is the displacement identity
	\[
	A-S_\theta A S_\theta^*
	=
	\psi\otimes k_0^\theta+k_0^\theta\otimes \chi,
	\]
	which gives an algebraic characterization of truncated Toeplitz operators. More recently, Gu and Ma obtained an analogous characterization for truncated Hankel operators on model spaces \cite{Gu:Ma}. These results show that the Toeplitz and Hankel structures on $\clk_\theta$ are encoded by rank-two defect relations involving the compressed shift $S_\theta$ and the kernels $k_0^\theta$ and $\widetilde k_0^\theta$.
	
	Another motivation for the present work comes from paired operators. Given complementary orthogonal projections $P$ and $Q$, a paired operator has the form
	\[
	AP+BQ.
	\]
	Such operators naturally arise when one studies multiplication operators relative to a fixed orthogonal decomposition. Recent work of Das, Das, and Sarkar \cite{Das:Das:Sarkar} developed algebraic characterizations of paired operators and related them to sums of Toeplitz and Hankel operators. The present paper develops an analogous framework adapted to model space decompositions.
	
	Our first object of study is the decomposition
	\[
	L^2(\T)=\clk_\theta\oplus \clk_\theta^\perp.
	\]
	For bounded symbols $\vp,\psi\in L^\infty(\T)$, we consider the $\clk_\theta$-paired multiplication operator
	\[
	X_{\vp,\psi}^{(\theta)}
	=
	M_\vp P_{\clk_\theta}+M_\psi(I-P_{\clk_\theta}).
	\]
	We prove that such operators are characterized by a finite-rank commutator identity with the bilateral shift $M_z$ on $L^2(\T)$. More precisely, after writing $a=\vp-\psi$, the commutator $[X_{\vp,\psi}^{(\theta)},M_z]$ is determined by the rank-two defect coming from the projection $P_{\clk_\theta}$.
	
	We then pass to the finer orthogonal decomposition
	\[
	L^2(\T)=\theta H^2\oplus \clk_\theta\oplus H^2_{-}.
	\]
	This leads to triply paired multiplication operators of the form
	\[
	X_{\vp,\psi,\chi}^{(\theta)}
	=
	M_\vp P_{\theta H^2}
	+
	M_\psi P_{\clk_\theta}
	+
	M_\chi P_{-}.
	\]
	We characterize these operators by a first-order commutator identity with $M_z$ and also by a spatial recovery formula. The latter recovers the three symbols from the action of the operator on the three summands. The main new feature is the rank-one term (cf. Theorem \ref{thm:triply-spatial})
	\[
	w \otimes \widetilde k_0^\theta,
	\]
	which appears because multiplication by $z$ does not leave $\clk_\theta$ invariant.
	
	In the final part of the paper, we apply this framework to sums of truncated Toeplitz and truncated Hankel operators on $\clk_\theta$. Let $\clt_\theta$ denote the space of truncated Toeplitz operators and $\clh_\theta$ the space of truncated Hankel operators on $\clk_\theta$. The sum space
	\[
	\cls_\theta=\clt_\theta+\clh_\theta
	\]
	is the model space analogue of the Toeplitz-plus-Hankel matrix class. Before turning to second-order displacements, we record cross-commutator reformulations of the Sarason and Gu--Ma identities (cf. Lemmas \ref{lem:sarason-equiv} and \ref{lem:hankel-equiv}), expressing the truncated Toeplitz and truncated Hankel structures as rank-two commutator defects with the compressed shift $S_\theta$, in the spirit of the classical Brown-Halmos commutator viewpoint.  Since the Toeplitz and Hankel displacement maps
	\[
	D_T(X)=X-S_\theta X S_\theta^*,
	\qquad
	D_H(X)=X-S_\theta^*XS_\theta^*
	\]
	need not commute, it is natural to consider the second-order displacement $D_HD_T$. For an operator
	\[
	T=A+B,\qquad A\in\clt_\theta,\quad B\in\clh_\theta,
	\]
	we compute $D_HD_T(T)$ explicitly and show that it is a finite-rank operator, of rank at most $4$, determined by the kernel functions $k_0^\theta$ and $\widetilde k_0^\theta$.
	
	We also give a displacement formulation of the condition that an operator $T\in\clb(\clk_\theta)$ be expressible as a sum of a truncated Toeplitz operator and a truncated Hankel operator. This is done in terms of the displacement maps $D_T$ and $D_H$; see Theorem~\ref{thm:TTO-THO-characterization}. A key point in the argument is the injectivity of these displacement maps on $\clb(\clk_\theta)$; see Lemma~\ref{lem:inj}.
	
	We also describe the non-uniqueness of the decomposition. If
	\[
	T=A+B=A'+B',
	\]
	where $A,A'\in\clt_\theta$ and $B,B'\in\clh_\theta$, then
	\[
	A-A'=B'-B\in\cli_\theta,
	\qquad
	\cli_\theta=\clt_\theta\cap\clh_\theta.
	\]
	Thus the decomposition is unique only modulo the intersection $\cli_\theta$. In the special case $\theta(z)=z^n$, this ambiguity space is two-dimensional and is spanned by the even and odd checkerboard matrices. In the same finite-dimensional setting, $\cls_\theta$ coincides with the classical space of Toeplitz-plus-Hankel matrices, and the Bevilacqua--Bonanni--Bozzo recurrence is recovered.
	
	The paper is organized as follows. Section~2 recalls the necessary preliminaries on Hardy spaces, model spaces, reproducing kernels, compressed shifts, and truncated Toeplitz and truncated Hankel operators. Section~3 studies $\clk_\theta$-paired multiplication operators and gives their commutator characterization. Section~4 treats triply paired multiplication operators associated with the decomposition
	\[
	L^2(\T)=\theta H^2\oplus\clk_\theta\oplus H^2_{-}.
	\]
	Section~5 applies the displacement framework to sums of truncated Toeplitz and truncated Hankel operators, proves the second-order finite-rank displacement formula, describes the ambiguity of the decomposition, and recovers the finite-dimensional Toeplitz-plus-Hankel matrix characterization when $\theta=z^n$.

	\section{Preliminaries}\label{sec:prelim}
	
	This section collects the analytic and operator-theoretic facts needed in the sequel. The conventions on Hardy spaces, inner functions, and model spaces follow~\cite{Sarason2007, NF}; for general background on multiplication, Toeplitz, and Hankel operators we refer to~\cite{Ni, Peller:Book}.

	Let $\clh$ be a complex separable Hilbert space, and let $\clb(\clh)$ denote the algebra of all bounded linear operators on $\clh$. For $A,B\in\clb(\clh)$, we denote their commutator by
	\[
	[A,B]=AB-BA.
	\] 
	We shall use the following notation for rank-one operators. For $p,q\in\clh$, define
	\[
	(p\otimes q)f=\langle f,q\rangle p,\qquad f\in\clh.
	\]
	For bounded operators $A,C\in\clb(\clh)$ and $\alpha\in\C$, we have
	\begin{equation}\label{eq:tensrule}
		A(p\otimes q)C=(Ap)\otimes(C^*q),\qquad
		(\alpha p)\otimes q=\alpha(p\otimes q),\qquad
		p\otimes(\alpha q)=\overline{\alpha}(p\otimes q).
	\end{equation}
	
	\subsection{Lebesgue and Hardy spaces}
	Let $\D = \{z \in \C : |z| < 1\}$ be the open unit disk and let $\T = \{z \in \C : |z| = 1\}$ be the unit circle. We equip $\T$ with the normalized Lebesgue measure:
	\[
	dm(t) = \frac{1}{2\pi} dt, \quad z = e^{it}, \ t \in [0, 2\pi).
	\]
	We write $L^2 := L^2(\T)$ for the Hilbert space of square integrable measurable functions on $\T$. 
	The inner product is given by
	\begin{align*}
		\ip{f}{g}_{L^2} = \int_{\T} f(z)\overline{g(z)}\,dm(z) = \frac{1}{2\pi} \int_{0}^{2\pi} f(e^{it})\overline{g(e^{it})}\,dt.
	\end{align*}
	The collection of functions $\{e_n(z) = z^n\}_{n \in \Z}$ forms a complete orthonormal basis for $L^2$. Any function $f \in L^2$ admits a formal Fourier expansion $f(z) = \sum_{n=-\infty}^\infty \hat{f}(n) z^n$, where convergence is in the $L^2$-norm.
	
	The Hardy space $H^2:=H^2(\T)$ is defined as the closed topological subspace of $L^2$ consisting of functions whose negative Fourier coefficients identically vanish:
	\[
	H^2 = \left\{ f \in L^2 : \hat{f}(n) = \int_{\T} f(z) z^{-n} dm(z) = 0 \text{ for all } n < 0 \right\}.
	\]
	By Fatou's Theorem, $H^2$ can be isometrically identified with the space of holomorphic functions on $\D$ satisfying 
	\[
	\sup_{0 < r < 1} \frac{1}{2\pi} \int_{0}^{2\pi} |f(re^{it})|^2 dt < \infty.
	\]
	
	We denote by $P_{+}$ the orthogonal Szeg\H{o} projection of $L^2$ onto $H^2$. If $f(z) = \sum_{n=-\infty}^\infty \hat{f}(n) z^n$, then
	\[
	P_{+} f(z) = \sum_{n=0}^\infty \hat{f}(n) z^n,
	\]
	and $P_{-} = I - P_{+}$ the projection onto the co-analytic space $H^2_{-} := \overline{zH^2} = L^2 \ominus H^2$.
	
	Let $H^\infty(\D)$ be the Banach algebra of all bounded analytic functions on the open unit disk $\D$, equipped with the norm
	\[
	\|f\|_\infty=\sup_{z\in\D}|f(z)|.
	\]
	That is,
	\[
	H^\infty(\D)
	=
	\left\{
	f:\D\to\C \;:\; f \text{ is analytic and } \sup_{z\in\D}|f(z)|<\infty
	\right\}.
	\]

	\subsection{Multiplication, Toeplitz, and Hankel operators}
	For $\vp\in L^\infty(\T)$, the Laurent multiplication operator $M_\vp : L^2 \to L^2$ is defined by 
	\[M_\vp f=\vp f,\quad f\in L^2.\] It is bounded and satisfies $\|M_\vp\|=\|\vp\|_\infty$. 
	
	The Toeplitz operator $T_\vp : H^2 \to H^2$ and Hankel operator $H_\vp : H^2 \to H^2_{-}$ with symbol $\vp$ are defined respectively as
	\[
	T_\vp f = P_{+} (\vp f), \qquad H_\vp f = P_{-} (\vp f), \quad f \in H^2.
	\]
	
	Let us rigorously record the action of the forward shift. The shift on $L^2$ is the unitary multiplication operator $M_z$. Restricted to the invariant subspace $H^2$, this becomes the unilateral isometric shift $T_z = \restr{M_z}{H^2}$. The adjoint of the unilateral shift is the backward shift $T_z^*$, whose action is defined by
	\[
	(T_z^* f)(z) = \frac{f(z) - f(0)}{z},\quad f\in H^2.
	\]
	
	\begin{theorem}\cite{BH}
		A bounded linear operator $A \in \clb(H^2)$ is a Toeplitz operator ($A = T_\vp$ for some $\vp \in L^\infty$) if and only if it satisfies 
		\begin{equation}
			T_z^* A T_z = A.
		\end{equation}
		Moreover, the correspondence $\vp \mapsto T_\vp$ is an isometric mapping, that is, $\|T_\vp\| = \|\vp\|_\infty$.
	\end{theorem}

	\subsection{Model spaces and operators}

	An analytic function $\theta\in H^\infty(\D)$ is called \emph{inner} if its radial boundary values satisfy
	\[
	|\theta(e^{it})|=1
	\quad\text{for almost every } t\in[0,2\pi).
	\]
	By the Maximum Modulus Theorem, every inner function satisfies $|\theta(z)|\le 1$ for every $z\in\D$.
	Moreover, if $\theta$ is non-constant, then
	\[
	|\theta(z)|<1,\qquad z\in\D.
	\]
	
	A fundamental result of classical Hardy space theory is Beurling's Theorem (cf. \cite{Beurling}), which states that every nontrivial closed $T_z$-invariant subspace of $H^2$ is of the form $\theta H^2$, where $\theta$ is an inner function. The associated \emph{model space} is defined as the orthogonal complement of this invariant subspace within the Hardy space:
	\[
	\clk_\theta = H^2 \ominus \theta H^2.
	\]
	Since
	\[
	L^2=H^2\oplus H^2_{-}
	\qquad\text{and}\qquad
	H^2=\clk_\theta\oplus \theta H^2,
	\]
	we have the orthogonal decomposition
	\[
	L^2=\theta H^2 \oplus  \clk_\theta  \oplus H^2_{-}.
	\]
	Consequently, the orthogonal complement of $\clk_\theta$ in $L^2$ is
	\[
	\clk_\theta^\perp=  \theta H^2 \oplus H^2_{-}.
	\]
	We denote by $P_{\clk_\theta}$ the orthogonal projection from $L^2$ onto $\clk_\theta$. Since $P_{\theta H^2} = M_\theta P_{+} M_{\bar{\theta}}$, we have
	\[
	P_{\clk_\theta} = P_{+} - P_{\theta H^2} = P_{+} - M_\theta P_{+} M_{\bar{\theta}}.
	\]
	
	The model space $\clk_\theta$ is a reproducing kernel Hilbert space. For any point $\lambda \in \D$, point evaluation is bounded on $\clk_\theta$, and the reproducing kernel for $\clk_\theta$ is given by
	\[
	k_\lambda^\theta(z) = \frac{1 - \overline{\theta(\lambda)}\theta(z)}{1 - \bar{\lambda}z}, \quad z \in \D.
	\]
	Thus, for any $f \in \clk_\theta$, we have $f(\lambda) = \langle f, k_\lambda^\theta \rangle$. For $\lambda = 0$, we have $k_0^\theta(z) = 1 - \overline{\theta(0)}\theta(z)$.
	
	The \emph{compressed shift} $S_\theta : \clk_\theta \to \clk_\theta$ is defined by 
	\[
	S_\theta f = P_{\clk_\theta}(zf), \quad f \in \clk_\theta.
	\]
	
	For a non-constant inner $\theta$, the compressed shift $S_{\theta}$ is a completely non-unitary contraction of class $C_{00}$ \cite[Ch.~III]{NF}, so $S_{\theta}^{n}f\to0$ and $(S_{\theta}^{*})^{n}f\to0$ for every
	$f\in \clk_{\theta}$.
	
	A critical structural tool on the model space is the \emph{canonical conjugation} $C_\theta : L^2 \to L^2$, defined in terms of boundary functions by
	\[
	(C_\theta f)(z) = \theta(z)\,\bar{z} \overline{f(z)},\quad z\in \T.
	\]
	The operator $C_\theta$ is a conjugate-linear, isometric involution ($C_\theta^2 = I$) that leaves the model space $\clk_\theta$ invariant. The \emph{conjugate kernel}, denoted $\tilde{k}_\lambda^\theta$ is given by
	\[
	\tilde{k}_\lambda^\theta(z) = (C_\theta k_\lambda^\theta)(z) = \frac{\theta(z) - \theta(\lambda)}{z - \lambda}.
	\]
	In particular, $\widetilde k_0^\theta = C_\theta k_0^\theta = T_z^*\theta$.
	
	From (\cite[Lemmas 2.2, 2.4]{Sarason2007}), we record the following identities
	\begin{align}
		I - S_\theta S_\theta^* &= k_0^\theta \otimes k_0^\theta, \label{eq:defect_forward} \\
		I - S_\theta^* S_\theta &= \tilde{k}_0^\theta \otimes \tilde{k}_0^\theta, \label{eq:defect_backward} \\
		S_\theta^* k_0^\theta &= -\overline{\theta(0)}\tilde{k}_0^\theta, \label{eq:asymA}\\
		S_\theta \tilde{k}_0^\theta &= -\theta(0) k_0^\theta. \label{eq:asymB}
	\end{align}

	For $\vp\in L^\infty(\T)$, the \emph{truncated Toeplitz operator} (TTO) is defined by
	\[
	A_\vp=\restr{P_{\clk_\theta}M_\vp}{\clk_\theta},
	\]
	and the \emph{dual truncated Toeplitz operator} (cf. \cite{Bhuia:Nag, Bhuia:Golla:Nag, Camara:Ross, Sang:Ding, Gu}) is defined by
	\[ D_\vp = \restr{P_{\clk_\theta^\perp}M_\vp}{\clk_\theta^\perp}.\]

	We also recall the corresponding truncated Hankel operator. Let $J:L^2\to L^2$ be the unitary operator defined by
	\[
	(Jf)(z)=\bar{z} f(\bar{z}), \qquad z\in\T.
	\]
	For $\vp\in L^\infty(\T)$, the \emph{truncated Hankel operator} (THO) with symbol $\vp$ is defined by
	\[
	B_\vp=\restr{P_{\clk_\theta}JP_{H^2_{-}}M_\vp}{\clk_\theta}.
	\]
	Throughout the paper we shall use the abbreviations TTO and THO interchangeably with their full names.	
	
	We introduce the following notations:
	
	\[
	\clt_\theta := \{A \in \clb(\clk_\theta) : A \text{ is a truncated Toeplitz operator on } \clk_\theta\},
	\]
	\[
	\clh_\theta := \{B \in \clb(\clk_\theta) : B \text{ is a truncated Hankel operator on } \clk_\theta\},
	\]
	and write
	\[
	\cls_\theta := \clt_\theta + \clh_\theta,\qquad
	\cli_\theta := \clt_\theta \cap \clh_\theta.
	\]

	We shall use the following algebraic characterizations of TTOs and THOs. By Sarason's theorem \cite[Theorem 4.1]{Sarason2007} (see also \cite{CFT}), an operator $A\in\clb(\clk_\theta)$ belongs to the space $\clt_\theta$ of bounded truncated Toeplitz operators if and only if there
	exist $\psi,\chi\in\clk_\theta$ such that
	\begin{equation}\label{eq:sarason}
		A-S_\theta AS_\theta^*
		=
		\psi\otimes k_0^\theta+k_0^\theta\otimes\chi .
	\end{equation}
	When $A=A_\vp$ with $\vp\in L^\infty(\T)$, the kernel of the bounded-symbol map is given by
	\[
	A_\vp=0
	\quad\Longleftrightarrow\quad
	\vp\in \theta H^2+\overline{\theta H^2}.
	\]
	
	By the Gu--Ma theorem \cite[Theorem 3.1]{Gu:Ma}, an operator $B\in\clb(\clk_\theta)$ belongs to the space $\clh_\theta$ of bounded truncated Hankel operators if and only if there exist $f,g\in\clk_\theta$
	such that
	\begin{equation}\label{eq:guma}
		B-S_\theta^*BS_\theta^*
		=
		f\otimes k_0^\theta+\widetilde{k}_0^\theta\otimes g .
	\end{equation}

	\section{Algebraic characterizations of doubly paired operators}
	
	The purpose of this section is to develop an algebraic characterization of paired multiplication operators associated with the model space decomposition
	\[
	L^2=\clk_\theta\oplus\clk_\theta^\perp .
	\]
	We begin with an abstract reduction result for paired operators associated with a pair of complementary orthogonal projections. This elementary block-matrix observation identifies the off-diagonal terms that obstruct the two summands from reducing the paired operator. We then specialize to the model space $\clk_\theta=H^2\ominus\theta H^2$ and obtain a commutator characterization of $\clk_\theta$-paired multiplication operators in terms of the bilateral shift $M_z$. The resulting identity shows that such operators are determined, up to a Laurent multiplication operator, by a rank-two defect term coming from the projection $P_{\clk_\theta}$. Finally, we record consequences for reducing decompositions and for analytic symbols.

	\begin{definition}\label{def:abstract-paired}
		Let $\clh$ be a Hilbert space, and let $P$ and $Q$ be orthogonal projections
		on $\clh$ such that
		\[
		P+Q=I_{\clh}.
		\]
		Thus $P$ and $Q$ are complementary orthogonal projections, equivalently
		$Q=I_{\clh}-P$ and $PQ=QP=0$. For $A,B\in\clb(\clh)$, the operator
		\[
		S_{A,B}:=AP+BQ
		\]
		is called the paired operator associated with $A,B$ and the pair $(P,Q)$.
		Equivalently, $S_{A,B}$ is associated with the orthogonal decomposition
		\[
		\clh=P\clh\oplus Q\clh .
		\]
	\end{definition}
	
	\begin{theorem}\label{thm:paired_reduction}
		Let $P$ and $Q$ be complementary orthogonal projections on a Hilbert space $\clh$, and let
		\[
		X=S_{A,B}=AP+BQ
		\]
		be the paired operator associated with $A,B\in\clb(\clh)$. Then the subspaces $P\clh$ and $Q\clh$ reduce $X$ if and only if
		\[
		QAP=0 \qquad\text{and}\qquad PBQ=0.
		\]
		In this case, with respect to the orthogonal decomposition $\clh=P\clh\oplus Q\clh$, the operator $X$ has the diagonal block representation
		\[
		X=
		\begin{bmatrix}
			PAP & 0\\
			0 & QBQ
		\end{bmatrix}.
		\]
		In particular, the above reduction holds whenever $A$ commutes with $P$ and $B$ commutes with $Q$.
	\end{theorem}
	
	\begin{proof}
		With respect to the decomposition
		\[
		\clh=P\clh\oplus Q\clh,
		\]
		we write the block matrix of $X\in\clb(\clh)$ as
		\[
		X=
		\begin{bmatrix}
			PXP & PXQ\\
			QXP & QXQ
		\end{bmatrix},
		\]
		where, for instance, $PXP$ is regarded as an operator from $P\clh$ into $P\clh$, and $PXQ$ as an operator from $Q\clh$ into $P\clh$. Since $X=AP+BQ$, using $P^2=P$, $Q^2=Q$, and $PQ=QP=0$, we obtain
		\begin{align*}
			PXP &= P(AP+BQ)P=PAP,\\
			PXQ &= P(AP+BQ)Q=PBQ,\\
			QXP &= Q(AP+BQ)P=QAP,\\
			QXQ &= Q(AP+BQ)Q=QBQ.
		\end{align*}
		Hence
		\[
		X=
		\begin{bmatrix}
			PAP & PBQ\\
			QAP & QBQ
		\end{bmatrix}.
		\]
		The subspaces $P\clh$ and $Q\clh$ reduce $X$ if and only if this block matrix is diagonal, which is equivalent to the vanishing of the two off-diagonal blocks. Therefore,
		\[
		P\clh \ \text{and}\ Q\clh \ \text{reduce } X
		\quad\Longleftrightarrow\quad
		QAP=0 \ \text{and}\ PBQ=0.
		\]
		When these conditions hold, the stated diagonal block representation follows immediately.
		
		Finally, suppose that $A$ commutes with $P$ and $B$ commutes with $Q$.
		Then
		\[
		QAP=QPA=0
		\]
		and
		\[
		PBQ=PQB=0.
		\]
		Thus the off-diagonal blocks vanish, and the reduction follows.
	\end{proof}

	\begin{definition}\label{def:Ktheta-paired}
		Let $\theta$ be a non-constant inner function. For bounded symbols $\vp,\psi\in L^\infty(\T)$, the \emph{$\clk_\theta$-paired operator} is defined as
		\[
		X^{(\theta)}_{\vp,\psi} := M_\vp P_{\clk_\theta} + M_\psi Q_\theta \quad\text{on }L^2,
		\]
		where $P_{\clk_\theta}$ is the orthogonal projection onto the model space $\clk_\theta = H^2 \ominus \theta H^2$ and $Q_\theta = I - P_{\clk_\theta}$ is the complementary orthogonal projection.
	\end{definition}
	
	Since $Q_\theta = I - P_{\clk_\theta}$, we may rewrite this operator as
	\begin{equation}\label{eq:Ktheta-decomp}
		X^{(\theta)}_{\vp,\psi} = M_\psi + M_{\vp-\psi}P_{\clk_\theta}.
	\end{equation}

	We shall use the following elementary multiplier criterion (cf. \cite{GMR}) several times in the sequel.
	
	\begin{lemma}\label{lem:multiplier-criterion}
		Let $\rho\in L^2$. Suppose there exists a constant $M>0$ such that
		\[
		\|\rho p\|_{L^2}\le M\|p\|_{L^2}
		\]
		for every trigonometric polynomial $p$. Then $\rho\in L^\infty(\T)$ and
		\[
		\|\rho\|_\infty\le M.
		\]
	\end{lemma}
	
		
		

	We shall also use the following standard description of the commutant of the bilateral shift.
	
	\begin{lemma}\cite[Lemma 4.17]{GMR}\label{lem:commutant-bilateral}
		If $D\in\clb(L^2)$ satisfies
		\[
		DM_z=M_zD,
		\]
		then there exists $\rho\in L^\infty(\T)$ such that
		\[
		D=M_\rho.
		\]
	\end{lemma}
	

	\begin{lemma} \label{lem:projection-commutators}
		For every inner function $\theta$,
		\[
		[P_{+},M_z]=1\otimes \bar{z}, \quad [P_{\theta H^2},M_z]=\theta\otimes \bar{z}\theta,
		\]
		and hence
		\[
		[P_{\clk_\theta},M_z]=1\otimes \bar{z}-\theta\otimes \bar{z}\theta.
		\]
	\end{lemma}
	
	\begin{proof}
		For $f=\sum_{n\in\Z}\widehat f(n)z^n\in L^2$,
		\[
		P_{+}(zf)-zP_{+}f =	\widehat f(-1)\cdot 1 =\ip{f}{\bar{z}}1.
		\]
		Hence
		\[
		[P_{+},M_z]=1\otimes \bar{z}.
		\]
		Next,
		\[
		P_{\theta H^2}=M_\theta P_{+}M_{\overline\theta}.
		\]
		Since $M_\theta$ and $M_z$ commute,
		\[
		\begin{aligned}
			\relax [P_{\theta H^2},M_z] 
			= M_\theta[P_{+},M_z]M_{\overline\theta}  
			=M_\theta(1\otimes \bar{z})M_{\overline\theta}  
			= \theta\otimes \bar{z}\theta.
		\end{aligned}
		\]
		Finally,
		\[
		P_{\clk_\theta}=P_{+}-P_{\theta H^2},
		\]
		and the desired formula follows.
	\end{proof}
	
	The following theorem gives an algebraic characterization of $\clk_\theta$-paired multiplication operators in terms of their commutator with the bilateral shift.

	\begin{theorem}\label{thm:two-component}
		Let $\theta$ be a non-constant inner function and let $a\in L^\infty(\T)$. For $X\in \clb(L^2)$, the following are equivalent:
		\begin{enumerate}
			\item There exists $\beta\in L^\infty(\T)$ such that
			\[
			X=M_\beta+M_aP_{\clk_\theta}.
			\]
			\item The commutator of $X$ with $M_z$ satisfies
			\[
			[X,M_z] = a\otimes\bar{z}-a\theta\otimes\bar{z}\theta.
			\]
		\end{enumerate}
		Consequently, $X$ is of the form
		\[
		X=M_\vp P_{\clk_\theta}+M_\psi Q_\theta
		\]
		for some $\vp,\psi\in L^\infty(\T)$ with $\vp-\psi=a$ if and only if the above commutator identity holds.
	\end{theorem}
	
	\begin{proof}
		Assume first that
		\[
		X=M_\beta+M_aP_{\clk_\theta}.
		\]
		Since multiplication operators commute with $M_z$,
		\[
		[X,M_z] =M_a[P_{\clk_\theta},M_z].
		\]
		By Lemma \ref{lem:projection-commutators},
		\[
		[P_{\clk_\theta},M_z]= 	1\otimes\bar{z}-\theta\otimes\bar{z}\theta.
		\]
		Therefore,
		\[
		[X,M_z]= M_a(1\otimes\bar{z}-\theta\otimes\bar{z}\theta)
		=a\otimes\bar{z}-a\theta\otimes\bar{z}\theta.
		\]
		Conversely, suppose that
		\[
		[X,M_z] = a\otimes\bar{z}-a\theta\otimes\bar{z}\theta.
		\]
		Let $Y=M_aP_{\clk_\theta}.$ By the first part of the proof,
		\[
		[Y,M_z]= a\otimes\bar{z}-a\theta\otimes\bar{z}\theta.
		\]
		Hence
		\[
		[X-Y,M_z]=0.
		\]
		By Lemma \ref{lem:commutant-bilateral}, there exists $\beta\in L^\infty(\T)$ such that $X-Y=M_\beta.$
		Thus
		\[
		X=M_\beta+M_aP_{\clk_\theta}.
		\]
		Finally, if $X=M_\vp P_{\clk_\theta}+M_\psi Q_\theta$, then
		\[
		X=M_\psi+M_{\vp-\psi}P_{\clk_\theta},
		\]
		so $a=\vp-\psi$. Conversely, if $X=M_\beta+M_aP_{\clk_\theta}$,
		then
		\[
		X=M_{\beta+a}P_{\clk_\theta}+M_\beta Q_\theta.
		\]
		This completes the proof.
	\end{proof}

	\begin{corollary}\label{cor:analytic-symbol-rigidity}
		Let $\theta$ be a non-constant inner function, and let $\vp,\psi\in L^\infty(\T)$. Consider
		\[
		X_{\vp,\psi}^{(\theta)} = M_\vp P_{\clk_\theta} + M_\psi(I-P_{\clk_\theta}) \quad\text{on }L^2.
		\]
		If the decomposition 
		\[
		L^2=\clk_\theta\oplus\clk_\theta^\perp
		\]
		reduces $X_{\vp,\psi}^{(\theta)}$, then both $\vp$ and $\psi$ are constant. Consequently,
		\[
		X_{\vp,\psi}^{(\theta)} = c_1\,P_{\clk_\theta}+c_2(I-P_{\clk_\theta})
		\]
		for some constants $c_1,c_2\in\C$.
	\end{corollary}
	
	\begin{proof}
		By Theorem~\ref{thm:paired_reduction}, the two summands $\clk_\theta$ and $\clk_\theta^\perp $ reduce	$X_{\vp,\psi}^{(\theta)}$ if and only if 
		\[
		(I-P_{\clk_\theta})M_\vp P_{\clk_\theta}=0 \qquad\text{and}\qquad
		P_{\clk_\theta}M_\psi(I-P_{\clk_\theta})=0.
		\]
		that is, to the invariance conditions $\vp\,\clk_\theta \subseteq \clk_\theta$ and $\psi\,\clk_\theta^\perp \subseteq \clk_\theta^\perp$.
		
		These two  conditions precisely say  that
		\[
		M_\vp(\clk_\theta)\subseteq \clk_\theta \qquad \text{and} \qquad  	M_\psi(\clk_\theta^\perp)\subseteq \clk_\theta^\perp, 
		\]
		or equivalently, $\vp\clk_\theta \subseteq\clk_\theta$ and $\psi \clk_\theta^\perp \subseteq\clk_\theta^\perp$.

		Now we show that $\vp\in H^\infty(\D)$. Since $k_0^\theta \in \clk_\theta$, the invariance condition yields $\vp\,k_0^\theta \in \clk_\theta\subseteq H^2$. Now $k_0^\theta=1-\overline{\theta(0)}\theta$, and since $\theta$ is non-constant the Maximum Modulus Theorem gives $|\theta(0)|<1$, so $|\overline{\theta(0)}\theta(z)|<1$ for $z\in\D$. Hence $k_0^\theta$ is invertible in $H^\infty(\D)$ with
		\[
		\frac{1}{k_0^\theta} =\frac{1}{1-\overline{\theta(0)}\theta}\in H^\infty(\D).
		\]
		Therefore,
		\[
		\vp=\bigl(\vp\,k_0^\theta\bigr)\cdot\frac{1}{k_0^\theta}\in H^2.
		\]
		Combined with $\vp\in L^\infty(\T)$, this gives $\vp\in H^2\cap L^\infty=H^\infty(\D)$. Since $\vp\in H^\infty(\D)$ and $\vp\,\clk_\theta\subseteq\clk_\theta$, \cite[Theorem~6.12]{GMR} implies that $\vp$ is constant.
		
		Next, we show that $\psi$ is constant. From $M_\psi\,\clk_\theta^\perp \subseteq\clk_\theta^\perp$, taking adjoints gives $M_{\bar\psi}\,\clk_\theta \subseteq\clk_\theta$, that is, $\bar\psi\,\clk_\theta \subseteq\clk_\theta$. Applying the preceding
		argument to $\bar\psi\in L^\infty(\T)$ yields $\bar\psi\in H^\infty(\D)$,
		and \cite[Theorem~6.12]{GMR} again forces $\bar\psi$ to be constant.
		Hence $\psi$ is constant.
		
		Therefore, both $\vp$ and $\psi$ are constant, and consequently
		\[
		X_{\vp,\psi}^{(\theta)}=c_1\,P_{\clk_\theta}+c_2(I-P_{\clk_\theta})
		\]
		for some $c_1,c_2\in\C$.
	\end{proof}

	\begin{remark}
		Corollary~\ref{cor:analytic-symbol-rigidity} shows that, under the analyticity assumption, a reducing $\clk_\theta$-paired operator is forced to lie in the two-dimensional algebra
		\[
		\operatorname{span}\{I,P_{\clk_\theta}\}
		=
		\{c_1P_{\clk_\theta}+c_2(I-P_{\clk_\theta}):c_1,c_2\in\C\}.
		\]
		Thus the conclusion is stronger than reduction with respect to
		\[
		L^2=\clk_\theta\oplus\clk_\theta^\perp;
		\]
		the operator acts as a scalar multiple of the identity on each of the two summands.
	\end{remark}

	\section{Algebraic characterizations of triply paired operators}

	In this section we pass from the two-component decomposition
	\[
	L^2=\clk_\theta\oplus\clk_\theta^\perp
	\]
	to the finer orthogonal decomposition
	\[L^2=\theta H^2\oplus \clk_\theta\oplus H^2_{-} .\]
	This gives rise to a three-symbol analogue of paired multiplication operators. The aim is to characterize these triply paired operators algebraically in terms of their commutator with the bilateral shift $M_z$, and then to derive a spatial identity that recovers the three symbols from the action of the operator on the three summands. The defect term in this identity is governed by the vector $\widetilde k_0^\theta$, reflecting the fact that multiplication by $z$ does not leave the model space $\clk_\theta$ invariant but instead produces a one-dimensional component in $\theta H^2$.

	\begin{definition}\label{def:triply-paired}
		Let $\theta$ be a non-constant inner function, and let $\vp,\psi,\chi\in L^\infty(\T)$. The triply paired multiplication operator associated with the decomposition
		\begin{equation}\label{eq:triple}
			L^2=\theta H^2\oplus\clk_\theta\oplus H_{-}^2   
		\end{equation}
		is defined by
		\[
		X_{\vp,\psi,\chi}^{(\theta)} = M_\vp P_{\theta H^2}+M_\psi P_{\clk_\theta}+M_\chi P_{-} .
		\]
	\end{definition}

	\begin{theorem} \label{thm:triply-characterization}
		Let $\theta$ be a non-constant inner function and let $a,b\in L^\infty(\T)$. For $X\in\clb(L^2)$, the following are equivalent:
		\begin{enumerate}
			\item There exists $\rho\in L^\infty(\T)$ such that
			\[
			X= M_{\rho+a+b}P_{\theta H^2} + M_{\rho+b}P_{\clk_\theta} +M_\rho P_{-}.
			\]
			\item The commutator of $X$ with $M_z$ satisfies
			\[
			[X,M_z] = a\theta\otimes\bar{z}\theta + b\otimes\bar{z}.
			\]
		\end{enumerate}
		Equivalently, $X$ is a triply paired multiplication operator
		\[
		X=M_\vp P_{\theta H^2}+ M_\psi P_{\clk_\theta}+ M_\chi P_{-}
		\]
		with
		\[
		a=\vp-\psi,
		\qquad
		b=\psi-\chi,
		\]
		if and only if
		\[
		[X,M_z]= (\vp-\psi)\theta\otimes\bar{z}\theta + (\psi-\chi)\otimes\bar{z}.
		\]
	\end{theorem}
	
	\begin{proof}
		Assume that $X= M_{\rho+a+b}P_{\theta H^2} + M_{\rho+b}P_{\clk_\theta} + M_\rho P_{-}.$ Using   $P_{+}=P_{\theta H^2}+P_{\clk_\theta}$ and $ I=P_{\theta H^2}+P_{\clk_\theta}+P_{-}$, we may rewrite $X$ as
		\[
		\begin{aligned}
			X &= M_\rho I+M_bP_{+} + M_aP_{\theta H^2}.
		\end{aligned}
		\]
		Thus, we have
		\[
		[X,M_z]= M_b[P_{+},M_z]+M_a[P_{\theta H^2},M_z],
		\]
		because multiplication operators commute with $M_z$. Now, by Lemma \ref{lem:projection-commutators},
		\[
		[P_{+},M_z]=1\otimes\bar{z}
		\quad \text{and}
		\quad 
		[P_{\theta H^2},M_z]=\theta\otimes\bar{z}\theta.
		\]
		Therefore,
		\[
		[X,M_z]=b\otimes\bar{z}+a\theta\otimes\bar{z}\theta.
		\]
		
		Conversely, suppose $[X,M_z]= a\theta\otimes\bar{z}\theta+b\otimes\bar{z}$. Let $Y=M_aP_{\theta H^2}+M_bP_{+}.$ Then, by the above computation,
		\[
		[Y,M_z]=a\theta\otimes\bar{z}\theta+b\otimes\bar{z}.
		\]
		Hence
		\[
		[X-Y,M_z]=0.
		\]
		By Lemma \ref{lem:commutant-bilateral}, there exists $\rho\in L^\infty(\T)$ such that $X-Y=M_\rho$.
		Thus
		\[
		X=M_\rho+M_bP_{+} + M_aP_{\theta H^2}.
		\]
		Using $P_{+}=P_{\theta H^2}+P_{\clk_\theta}$ and $I=P_{\theta H^2}+P_{\clk_\theta}+P_{-}$, we obtain
		\[
		\begin{aligned}
			X &= M_\rho(P_{\theta H^2}+P_{\clk_\theta}+P_{-}) + M_b(P_{\theta H^2}+P_{\clk_\theta})+ M_aP_{\theta H^2}  \\
			&= M_{\rho+a+b}P_{\theta H^2}+ M_{\rho+b}P_{\clk_\theta} + M_\rho P_{-}.
		\end{aligned}
		\]
		This proves the equivalence.
		
		Finally, putting
		\[
		\vp=\rho+a+b,\qquad \psi=\rho+b,\qquad \chi=\rho
		\]
		gives
		\[
		a=\vp-\psi, \qquad b=\psi-\chi.
		\]
		The final assertion follows.
	\end{proof}

	\begin{lemma}\label{lem:proj}
		For every $f\in \clk_{\theta}$, $P_{\theta H^2}(zf)=\langle f,\tilde{k}_0^{\theta}\rangle\,\theta$.
	\end{lemma}
	
	\begin{proof}
		We find the unique $v\in\theta H^2$ with $\langle zf,g\rangle=\langle v,g
		\rangle$ for all $g\in\theta H^2$. Write $g=\theta h$, $h\in H^2$, and
		$h=h(0)+zh_{1}$; then $h(0)=\langle h,1\rangle=\langle\theta h,\theta\rangle
		=\langle g,\theta\rangle$. Since $f\perp\theta H^2$,
		\[
		\langle zf,g\rangle=\langle f,\bar{z}\theta h\rangle
		=\overline{h(0)}\,\langle f,\bar{z}\theta\rangle
		=\langle\langle f,\bar{z}\theta\rangle\theta,g\rangle .
		\]
		Thus $P_{\theta H^2}(zf)=\langle f,\bar{z}\theta\rangle\theta$. To rewrite this
		in terms of $\tilde{k}_0^{\theta}$, we note that for any $f\in H^{2}$ one has $\langle f,
		\bar{z}\rangle=0$. Hence
		\[
		\langle f,\bar{z}\theta\rangle
		=\langle f,\bar{z}(\theta-\theta(0))\rangle
		=\langle f,\tilde{k}_0^{\theta}\rangle,
		\]
		since on $\T$ one has $\bar{z}(\theta-\theta(0))=(\theta-\theta(0))/z
		=T_z^*\theta=\tilde{k}_0^{\theta}$. Therefore, $P_{\theta H^2}(zf)=\langle f,\tilde{k}_0^{\theta}\rangle\theta$.
	\end{proof}

	\begin{theorem}\label{thm:triply-spatial}
		$X\in\clb (L^2)$ is a triply paired operator $X^{(\theta)}_{\vp,\psi,\chi}$ for some $\vp,\psi,\chi\in L^{\infty}(\T)$ if and only if there exists $w\in L^\infty(\T)$ with
		\begin{equation}\label{eq:tsp}
			X=M_z^* XM_zP_{+}+M_z XM_z^*P_{-}-\bigl(w\otimes \tilde{k}_0^{\theta}\bigr)P_{\clk_{\theta}} .
		\end{equation}
		When these hold, the symbols are recovered by
		\[
		\chi=zX(\bar{z}), \qquad \vp=\overline{\theta}\,X(\theta),\quad \psi=\vp-z\overline{\theta}\,w.
		\]
	\end{theorem}
	
	\begin{proof}
		$(\Rightarrow)$ Assume
		\[
		X=M_{\vp}P_{\theta H^2}+M_{\psi}P_{\clk_{\theta}}+M_{\chi}P_{-}.
		\]
		Recall that
		\[
		P_{+}=P_{\theta H^2}+P_{\clk_{\theta}}.
		\]
		
		First let $f\in H^2_{-}$. Since $H^2_{-}$ is invariant under $M_z^*$ and $\restr{X}{H^2_{-}}=M_\chi$, we have
		\[
		M_zXM_z^*f = M_zM_\chi M_z^*f=M_\chi f=Xf.
		\]
		Hence
		\[
		XP_{-}=M_zXM_z^*P_{-}.
		\]
		
		Next let $f\in \theta H^2$. Since $\theta H^2$ is invariant under $M_z$ and $\restr{X}{\theta H^2}=M_\vp$, we get
		\[
		M_z^*XM_zf =M_z^*M_\vp M_zf=M_\vp f=Xf.
		\]
		Thus
		\[
		XP_{\theta H^2}=M_z^*XM_zP_{\theta H^2}.
		\]
		
		Finally let $f\in \clk_\theta$. By Lemma~\ref{lem:proj},
		\[
		M_z f=S_\theta f+\langle f,\tilde k_0^\theta\rangle\theta,
		\]
		where $S_\theta f\in\clk_\theta$ and $\langle f,\tilde k_0^\theta\rangle\theta\in\theta H^2$. Therefore,
		\[
		XM_z f = M_\psi S_\theta f + \langle f,\tilde k_0^\theta\rangle M_\vp\theta.
		\]
		Since $S_\theta f=zf-\langle f,\tilde k_0^\theta\rangle\theta$, we obtain
		\[
		XM_z f =\psi zf+\langle f,\tilde k_0^\theta\rangle(\vp-\psi)\theta.
		\]
		Applying $M_z^*$ gives
		\[
		M_z^*XM_z f=\psi f+\overline z(\vp-\psi)\theta
		\langle f,\tilde k_0^\theta\rangle.
		\]
		But $Xf=\psi f$ for $f\in\clk_\theta$. Hence
		\[
		M_z^*XM_z f= Xf+w\langle f,\tilde k_0^\theta\rangle,
		\]
		where
		\[
		w:=(\vp-\psi)\overline z\,\theta\in L^\infty(\T).
		\]
		Equivalently,
		\[
		XP_{\clk_\theta} = M_z^*XM_zP_{\clk_\theta}- (w\otimes\tilde k_0^\theta)P_{\clk_\theta}.
		\]
		
		Combining the three identities over the orthogonal decomposition $L^2(\T)=\theta H^2\oplus\clk_\theta\oplus H^2_{-}$ yields
		\[
		X=M_z^*XM_zP_++M_zXM_z^*P_{-}-(w\otimes\tilde k_0^\theta)P_{\clk_\theta},
		\]
		which is precisely \eqref{eq:tsp}.

		$(\Leftarrow)$ Assume  the identity \eqref{eq:tsp} holds. We recover the three symbols $\chi,\vp,\psi$ by isolating the three components of the decomposition \eqref{eq:triple} in turn.
		
		Right-multiplying \eqref{eq:tsp} by $P_{-}$ gives
		\[
		XP_{-}=M_zXM_z^*P_{-}.
		\]
		Multiplying on the left by $M_z^*$, we obtain
		\[
		M_z^*XP_{-}=XM_z^*P_{-}.
		\]
		Thus, on $H^2_{-}$, the operator $X$ intertwines the restriction of $M_z^*$ to $H^2_{-}$. Since
		\[
		M_z^*\bar{z}^m=\bar{z}^{m+1},\qquad m\ge 1,
		\]
		it follows by induction that
		\[
		X(\bar{z}^m)=\bar{z}^{m-1}X(\bar{z}),\qquad m\ge 1.
		\]
		Set
		\[
		\alpha:=X(\bar{z})\in L^2,
		\qquad
		\chi:=z\alpha=zX(\bar{z}).
		\]
		Then
		\[
		X(\bar{z}^m)=\chi\,\bar{z}^m,\qquad m\ge 1.
		\]
		Hence, for every polynomial $p$ in $H^2_{-}$, we have
		\[
		Xp=\chi p.
		\]
		Moreover,
		\[
		\|\chi p\|=\|Xp\|\le \|X\|\,\|p\|.
		\]
		
		We now extend this estimate to arbitrary trigonometric polynomials. Let $r$ be a trigonometric polynomial. Choose $N$ sufficiently large so that $z^{-N}r\in H^2_{-}$. Then
		\[
		\|\chi r\|
		=
		\|z^{-N}\chi r\|
		=
		\|\chi z^{-N}r\|
		\le
		\|X\|\,\|z^{-N}r\|
		=
		\|X\|\,\|r\|.
		\]
		By Lemma~\ref{lem:multiplier-criterion}, $\chi\in L^\infty(\T)$. Since $X$ agrees with $M_\chi$ on the dense set of polynomials in $H^2_{-}$, we conclude that
		\[
		XP_{-}=M_\chi P_{-}.
		\]

		Again, right-multiplying \eqref{eq:tsp}
		by $P_{{\theta}H^2}$ gives $XP_{{\theta}H^2}=M_z^* XM_zP_{{\theta}H^2}$. 
		Hence $M_z XP_{{\theta}H^2}=XM_zP_{{\theta}H^2}$, and inducting on $n\ge0$ yields $X(z^{n}\theta)=z^{n}X(\theta)$. Set $\vp:=\overline{\theta}X(\theta)\in L^2$; since $|\theta|=1$ a.e.\
		on $\T$, $X(\theta)=\theta\vp$, and by linearity
		\begin{equation}\label{eq:Xthetap}
			X(\theta p)=\theta\vp\, p\qquad\text{for every analytic polynomial }p.
		\end{equation}
		We claim $\vp\in L^{\infty}(\T)$. By \eqref{eq:Xthetap},
		$\|\theta\vp p\|=\|X(\theta p)\|\le\|X\|\,\|\theta p\|=\|X\|\,\|p\|$
		for analytic polynomials $p$, and using $|\theta|=1$ a.e.,
		$\|\vp p\|\le\|X\|\,\|p\|$ for analytic $p$.
		
		To extend the inequality to all trigonometric polynomials, let
		\[
		p(z)=\sum_{k=-m}^{n} c_k z^k
		\]
		be arbitrary. Choose $N\ge m$. Then
		\[
		q(z):=z^N p(z)=\sum_{k=-m}^{n} c_k z^{k+N}
		\]
		is an analytic polynomial. Since multiplication by \(z^{-N}\) is unitary on $L^2(\T)$, we have $p=z^{-N}q$ and $\|p\|=\|q\|$. Therefore
		\[
		\|\vp p\| = \|\vp z^{-N}q\| =
		\|z^{-N}\vp q\| = \|\vp q\| \le \|X\| \|q\|= \|X\| \|p\|.
		\]
		By Lemma~\ref{lem:multiplier-criterion}, $\vp\in L^{\infty}(\T)$ with $\|\vp\|_{\infty}\le\|X\|$. Hence 
		\[
		XP_{{\theta}H^2}=M_{\vp}P_{{\theta}H^2}
		\]
		on the dense set of $\theta p$ with $p$ analytic, and extends uniquely by continuity.

		Now, right-multiplying \eqref{eq:tsp} by $P_{\clk_{\theta}}$ gives
		\begin{equation}\label{eq:XPK}
			XP_{\clk_{\theta}}=M_z^* XM_zP_{\clk_{\theta}}-(w\otimes  \tilde{k}_0^{\theta} )P_{\clk_{\theta}}.
		\end{equation}
		
		Since $w\in L^\infty(\T)$, $\overline z\theta$ is unimodular, and $\vp\in L^\infty(\T)$ by the preceding part, define
		\begin{equation}\label{eq:psi-def}
			\psi:=\vp-z\overline\theta\,w\in L^\infty(\T).
		\end{equation}
		Then
		\[
		w=(\vp-\psi)\overline z\theta.
		\]

		Define $D:=(X-M_{\psi}P_{\clk_{\theta}})P_{\clk_{\theta}}$. Then
		\begin{equation}\label{eq:Dvan}
			D(\theta H^2)=0, \quad \restr{D}{H^2_{-}}=0.
		\end{equation}
		We claim
		\begin{equation}\label{eq:Dident}
			D=M_z^* DM_zP_{\clk_{\theta}}.
		\end{equation}
		It suffices to check $Df=M_z^* DM_z f$ for every
		$f\in \clk_{\theta}$.

		Fix $f\in \clk_{\theta}$. By Lemma~\ref{lem:proj},
		$M_z f=  S_\theta  f+\langle f,\tilde{k}_0^{\theta}\rangle\theta$ with $  S_\theta  f\in \clk_{\theta}$ and
		$\langle f,\tilde{k}_0^{\theta}\rangle\theta\in \theta H^2$. Since $D$ vanishes on $\theta H^2$ by \eqref{eq:Dvan}, we have
		\begin{equation}\label{eq:Dstep1}
			DM_z f=D  S_\theta  f=(X-M_{\psi}P_{\clk_{\theta}})  S_\theta  f=X(  S_\theta  f)-\psi\,  S_\theta  f.
		\end{equation}
		Applying $M_z^*$, \begin{equation}\label{eq:Dstep2}
			M_z^* DM_z f=M_z^* X(  S_\theta  f)-\bar{z}\,\psi\,  S_\theta  f.
		\end{equation}
		
		To simplify $M_z^* X(  S_\theta  f)$, applying the spatial identity \eqref{eq:tsp} to $f$, we obtain
		\[
		X f=M_z^* XM_z f-\langle f,  \tilde{k}_0^{\theta} \rangle\,w.
		\]
		Decompose $M_z f$ as above; using $\restr{X}{\theta H^2}=M_{\vp}$ from above,
		\[
		XM_z f=X(  S_\theta  f)+\langle f,\tilde{k}_0^{\theta}\rangle\,X(\theta)
		=X(  S_\theta  f)+\langle f,\bar{z}\theta\rangle\,\vp\,\theta,
		\]
		by Lemma~\ref{lem:proj}. Writing
		$c:=\langle f,\tilde{k}_0^{\theta}\rangle=\langle f,\bar{z}\theta\rangle$ and substituting, we obtain 
		$X f=M_z^* X(  S_\theta  f)+c\,\bar{z}\,\vp\,\theta-c\,w,$ which gives
		\[
		M_z^* X(  S_\theta  f)=X f-c\,\bar{z}\,\vp\,\theta+c\,w.
		\]
		By using $w=(\vp-\psi)\bar{z}\theta$, we have
		\begin{equation}\label{eq:Dstep3}
			M_z^* X(  S_\theta  f)=X f-c\,\bar{z}\,\theta\,\psi.
		\end{equation}
		Finally, substituting \eqref{eq:Dstep3} into \eqref{eq:Dstep2} and using 
		$\bar{z}\,\psi\,  S_\theta  f =\bar{z}\,\psi\bigl(zf-c\,\theta\bigr)
		=\psi f-c\,\bar{z}\,\theta\,\psi$, and, therefore
		\[
		M_z^* DM_z f=X f-\psi f=(X-M_{\psi})f=(X-M_{\psi}P_{\clk_{\theta}})P_{\clk_{\theta}} f=Df,
		\]
		and hence this proves \eqref{eq:Dident}.

		For $f\in\clk_\theta$, iterating \eqref{eq:Dident} and using $M_zf=S_\theta f+\langle f,\widetilde k_0^\theta\rangle\theta$, together with $\restr{D}{\theta H^2}=0$, yields
		\[
		Df=(M_z^*)^nD S_\theta^n f,\qquad n\ge1.
		\]
		Since $M_z^*$ is unitary on $L^2$,
		\[
		\|Df\|=\|(M_z^*)^{n}D  S_\theta ^{n}f\|\le\|D\|\,\|  S_\theta ^{n}f\|.
		\]
		The compressed shift $  S_\theta $ lies in the class $C_{00}$ (cf. \cite[Proposition 4.3]{NF}), hence $  S_\theta ^{n}f\to0$ in norm as $n\to\infty$; thus $\|Df\|=0$, that is, $D=0$ on $ \clk_{\theta}$. Combined with \eqref{eq:Dvan},
		$D=0$ as an operator on $L^{2}$, so $XP_{\clk_{\theta}}=M_{\psi}P_{\clk_{\theta}}$ with $\psi\in L^\infty(\T)$ already fixed by \eqref{eq:psi-def}.
		
		Summing the three pieces over the orthogonal decomposition
		\eqref{eq:triple} gives $X=M_{\vp}P_{{\theta}H^2}+M_{\psi}P_{\clk_{\theta}}+M_{\chi}P_{-}
		=X^{(\theta)}_{\vp,\psi,\chi}$, as required.
	\end{proof}

	Thus the triply paired structure is encoded both by the first-order commutator identity in Theorem~\ref{thm:triply-characterization} and by the spatial recovery formula in Theorem~\ref{thm:triply-spatial}. The latter will be useful when comparing paired decompositions with Toeplitz--Hankel type decompositions on the model space.

	\section{Sum of truncated Toeplitz and truncated Hankel operators}
	
	The rank-two displacement identities of Sarason \cite{Sarason2007} for TTOs and of Gu--Ma \cite{Gu:Ma} for THOs (cf.\ Section~\ref{sec:prelim}) involve the compressed shift $S_\theta$ and the boundary kernels $k_0^\theta$ and $\wt{k}_0^\theta$. The Toeplitz and Hankel displacement maps do not commute, so the natural object for studying their sum is the
	\emph{second-order} displacement $D_H D_T$, which we develop in this section.

	\begin{lemma}\label{lem:sarason-equiv}
		For any bounded linear operator $A\in\clb(\clk_\theta)$, the identity
		\[
		A - S_\theta A S_\theta^*
		=
		\psi \otimes k_0^\theta + k_0^\theta \otimes \chi
		\]
		for some $\psi,\chi\in\clk_\theta$ holds if and only if the identity
		\[
		A S_\theta - S_\theta A
		=
		u \otimes \tilde{k}_0^\theta + k_0^\theta \otimes v
		\]
		holds for some $u,v\in\clk_\theta$. More precisely, one may choose
		\[
		u = -(S_\theta A \tilde{k}_0^\theta + \theta(0) \psi)
		\quad \text{and} \quad
		v = S_\theta^* \chi
		\]
		in one direction, and
		\[
		\psi = A k_0^\theta - \overline{\theta(0)} u
		\quad \text{and} \quad
		\chi = S_\theta v
		\]
		in the reverse direction.
	\end{lemma}
	
	\begin{proof}
		$(\Rightarrow)$ Assume $A - S_\theta A S_\theta^* = \psi \otimes k_0^\theta + k_0^\theta \otimes \chi$. Right-multiplying the entire equation by the forward shift $S_\theta$, and then applying the relations \eqref{eq:defect_backward}, \eqref{eq:asymA} yields
		\begin{equation}
			\begin{split}
				A S_\theta - S_\theta A (S_\theta^* S_\theta) &= (\psi \otimes k_0^\theta)S_\theta + (k_0^\theta \otimes \chi)S_\theta\\
				A S_\theta - S_\theta A (I - \tilde{k}_0^\theta \otimes \tilde{k}_0^\theta) &=\psi \otimes (S_\theta^* k_0^\theta) + k_0^\theta \otimes (S_\theta^* \chi)\\
				A S_\theta - S_\theta A + (S_\theta A \tilde{k}_0^\theta) \otimes \tilde{k}_0^\theta &=-\theta(0) \psi \otimes \tilde{k}_0^\theta + k_0^\theta \otimes (S_\theta^* \chi)\\
				A S_\theta - S_\theta A &= -(S_\theta A \tilde{k}_0^\theta + \theta(0) \psi) \otimes \tilde{k}_0^\theta + k_0^\theta \otimes (S_\theta^* \chi)\\
				&=u \otimes \tilde{k}_0^\theta + k_0^\theta \otimes v,
			\end{split}
		\end{equation}
		where $u = -(S_\theta A \tilde{k}_0^\theta + \theta(0) \psi)$ and $v = S_\theta^* \chi$.

		$(\Leftarrow)$ Conversely, assume $A S_\theta - S_\theta A = u \otimes \tilde{k}_0^\theta + k_0^\theta \otimes v$. Right-multiplying this equation by $S_\theta^*$ and then using the relations \eqref{eq:defect_forward}, \eqref{eq:asymB} yields
		\begin{equation}
			\begin{split}
				A (S_\theta S_\theta^*) - S_\theta A S_\theta^* &= u \otimes (S_\theta \tilde{k}_0^\theta) + k_0^\theta \otimes (S_\theta v)\\
				A (I - k_0^\theta \otimes k_0^\theta) - S_\theta A S_\theta^* &= u \otimes (S_\theta \tilde{k}_0^\theta) + k_0^\theta \otimes (S_\theta v)\\
				A - S_\theta A S_\theta^* &= (A k_0^\theta) \otimes k_0^\theta + u \otimes (S_\theta \tilde{k}_0^\theta) + k_0^\theta \otimes (S_\theta v)\\
				&= (A k_0^\theta) \otimes k_0^\theta - \overline{\theta(0)} u \otimes k_0^\theta + k_0^\theta \otimes (S_\theta v) \\
				&= (A k_0^\theta - \overline{\theta(0)} u) \otimes k_0^\theta + k_0^\theta \otimes (S_\theta v)\\
				&=\psi \otimes k_0^\theta + k_0^\theta \otimes \chi,
			\end{split}
		\end{equation}
		where $\psi = A k_0^\theta - \overline{\theta(0)} u$ and $\chi = S_\theta v$.
	\end{proof}

	\begin{lemma}\label{lem:hankel-equiv}
		For any bounded linear operator $B\in\clb(\clk_\theta)$, the identity
		\[
		B-S_\theta^*BS_\theta^*
		=
		x\otimes k_0^\theta+\tilde{k}_0^\theta\otimes y
		\]
		for some $x,y\in\clk_\theta$ holds if and only if the identity
		\[
		BS_\theta-S_\theta^*B
		=
		u\otimes\tilde{k}_0^\theta+\tilde{k}_0^\theta\otimes v
		\]
		holds for some $u,v\in\clk_\theta$. More precisely, one may choose
		\[
		u=-(S_\theta^*B\tilde{k}_0^\theta+\theta(0)x)
		\quad\text{and}\quad
		v=S_\theta^*y
		\]
		in one direction, and
		\[
		x=Bk_0^\theta-\overline{\theta(0)}u
		\quad\text{and}\quad
		y=S_\theta v
		\]
		in the reverse direction.
	\end{lemma}
	
	\begin{proof}
		$(\Rightarrow)$ Assume $B - S^*_\theta B S^*_\theta = x \otimes k_0^\theta + \tilde{k}_0^\theta \otimes y$. Right-multiplying the entire equation by the forward shift $S_\theta$  and  using the relations \eqref{eq:defect_backward}, \eqref{eq:asymA} yields
		\begin{equation}
			\begin{split}
				B S_\theta - S^*_\theta B (S^*_\theta S_\theta) &= (x \otimes k_0^\theta)S_\theta + (\tilde{k}_0^\theta \otimes y)S_\theta\\
				B S_\theta - S^*_\theta B (I - \tilde{k}_0^\theta \otimes \tilde{k}_0^\theta) &=x \otimes (S^*_\theta k_0^\theta) + \tilde{k}_0^\theta \otimes (S^*_\theta y)\\
				B S_\theta - S^*_\theta B + (S^*_\theta B \tilde{k}_0^\theta) \otimes \tilde{k}_0^\theta  &=-\theta(0) x \otimes \tilde{k}_0^\theta + \tilde{k}_0^\theta \otimes (S^*_\theta y)\\
				B S_\theta - S^*_\theta B &= -(S^*_\theta B \tilde{k}_0^\theta + \theta(0) x) \otimes \tilde{k}_0^\theta + \tilde{k}_0^\theta \otimes (S^*_\theta y)\\
				&=u \otimes \tilde{k}_0^\theta + \tilde{k}_0^\theta \otimes v,
			\end{split}
		\end{equation}
		where $u = -(S^*_\theta B \tilde{k}_0^\theta + \theta(0) x)$ and $v = S^*_\theta y$.
		
		$(\Leftarrow)$ Conversely, assume $B S_\theta - S^*_\theta B = u \otimes \tilde{k}_0^\theta + \tilde{k}_0^\theta \otimes v$. Right-multiplying by $S^*_\theta$ and using the relations \eqref{eq:defect_forward}, \eqref{eq:asymB}, we obtain 
		\begin{equation}
			\begin{split}
				B (S_\theta S^*_\theta) - S^*_\theta B S^*_\theta &= u \otimes (S_\theta \tilde{k}_0^\theta) + \tilde{k}_0^\theta \otimes (S_\theta v)\\
				B (I - k_0^\theta \otimes k_0^\theta) - S^*_\theta B S^*_\theta &= u \otimes (S_\theta \tilde{k}_0^\theta) + \tilde{k}_0^\theta \otimes (S_\theta v)\\
				B - S^*_\theta B S^*_\theta &= (B k_0^\theta) \otimes k_0^\theta + u \otimes (S_\theta \tilde{k}_0^\theta) + \tilde{k}_0^\theta \otimes (S_\theta v)\\
				B - S^*_\theta B S^*_\theta &= (B k_0^\theta) \otimes k_0^\theta - \overline{\theta(0)} u \otimes k_0^\theta + \tilde{k}_0^\theta \otimes (S_\theta v)\\
				&= (B k_0^\theta - \overline{\theta(0)} u) \otimes k_0^\theta + \tilde{k}_0^\theta \otimes (S_\theta v)\\
				&=x \otimes k_0^\theta + \tilde{k}_0^\theta \otimes y,
			\end{split}
		\end{equation}
		where $x = B k_0^\theta - \overline{\theta(0)} u$ and $y = S_\theta v$. This proves the reverse implication and completes the proof.
	\end{proof}

	We define the Toeplitz and Hankel first-order displacement operators $D_T, D_H : \clb(\clk_\theta) \to \clb(\clk_\theta)$ as follows:
	\begin{align*}
		D_T(X) &:= X - S_\theta X S_\theta^*, \\
		D_H(X) &:= X - S_\theta^* X S_\theta^*.
	\end{align*}

	\begin{lemma}\label{lem:inj}
		The maps $D_T$ and $D_H$ are injective on $\clb(\clk_\theta)$.
	\end{lemma}
	
	\begin{proof}
		We first prove the injectivity of $D_T$. Suppose that $D_T(X)=0$. Then
		\[
		X=S_\theta X S_\theta^*.
		\]
		Iterating this identity gives
		\[
		X=S_\theta^n X(S_\theta^*)^n,\qquad n\geq 1.
		\]
		Hence, for $f,g\in \clk_\theta$,
		\[
		\begin{aligned}
			\langle Xf,g\rangle
			&= \left\langle S_\theta^n X(S_\theta^*)^n f,g\right\rangle  \\
			&= \left\langle X(S_\theta^*)^n f,(S_\theta^*)^n g\right\rangle .
		\end{aligned}
		\]
		Therefore,
		\[
		|\langle Xf,g\rangle|
		\leq \|X\|\,\|(S_\theta^*)^n f\|\,\|(S_\theta^*)^n g\|.
		\]
		Since $S_\theta^*$ is strongly stable on $\clk_\theta$, we have
		\[
		(S_\theta^*)^n f\to 0 \qquad \text{and} \qquad (S_\theta^*)^n g\to 0.
		\]
		Thus $\langle Xf,g\rangle=0$ for all $f,g\in\clk_\theta$, and hence $X=0$.
		This proves that $D_T$ is injective.
		
		Next suppose that $D_H(X)=0$. Then
		\[
		X=S_\theta^* X S_\theta^*.
		\]
		By iteration,
		\[
		X=(S_\theta^*)^n X(S_\theta^*)^n,\qquad n\geq 1.
		\]
		Thus, for $f,g\in\clk_\theta$,
		\[
		\begin{aligned}
			\langle Xf,g\rangle
			&=
			\left\langle (S_\theta^*)^n X(S_\theta^*)^n f,g\right\rangle  \\
			&=
			\left\langle X(S_\theta^*)^n f,S_\theta^n g\right\rangle .
		\end{aligned}
		\]
		Consequently,
		\[
		|\langle Xf,g\rangle|
		\leq
		\|X\|\,\|(S_\theta^*)^n f\|\,\|S_\theta^n g\|.
		\]
		Since $S_\theta$ is a contraction, $\|S_\theta^n g\|\leq \|g\|$, while
		\[
		(S_\theta^*)^n f\to 0.
		\]
		It follows that $\langle Xf,g\rangle=0$ for all $f,g\in\clk_\theta$, and therefore, $X=0$. Hence $D_H$ is injective.
	\end{proof}

	\begin{lemma} \label{lem:commutator-displacements}
		For every $X\in\clb(\clk_\theta)$, we have
		\[
		[D_H,D_T](X) = k_0^\theta\otimes \left(S_\theta^2X^*k_0^\theta\right) - \widetilde k_0^\theta\otimes \left(S_\theta^2X^*\widetilde k_0^\theta\right),
		\]
		which has rank at most two.
	\end{lemma}
	
	\begin{proof}
		We explicitly expand the operator compositions on an arbitrary $X$:
		\begin{align*}
			(D_H D_T)(X) &= D_H(X - S_\theta X S_\theta^*) \\
			&= (X - S_\theta X S_\theta^*) - S_\theta^*(X - S_\theta X S_\theta^*)S_\theta^* \\
			&= X - S_\theta X S_\theta^* - S_\theta^* X S_\theta^* + S_\theta^* S_\theta X (S_\theta^*)^2,
		\end{align*}
		and
		\begin{align*}
			(D_T  D_H)(X) &= D_T(X - S_\theta^* X S_\theta^*) \\
			&= (X - S_\theta^* X S_\theta^*) - S_\theta(X - S_\theta^* X S_\theta^*)S_\theta^* \\
			&= X - S_\theta^* X S_\theta^* - S_\theta X S_\theta^* + S_\theta S_\theta^* X (S_\theta^*)^2.
		\end{align*}
		Therefore,
		\begin{align*}
			[D_H, D_T](X) &:= (D_H  D_T)(X) - (D_T D_H)(X) \\
			&= (S_\theta^* S_\theta - S_\theta S_\theta^*) X (S_\theta^*)^2.
		\end{align*}
		Substituting $S_\theta^* S_\theta = I - \tilde{k}_0^\theta \otimes \tilde{k}_0^\theta$ and $S_\theta S_\theta^* = I - k_0^\theta \otimes k_0^\theta$, we get $S_\theta^* S_\theta - S_\theta S_\theta^*=k_0^\theta \otimes k_0^\theta - \tilde{k}_0^\theta \otimes \tilde{k}_0^\theta$. Thus
		\begin{align*}
			[D_H, D_T](X) &= (k_0^\theta \otimes k_0^\theta - \tilde{k}_0^\theta \otimes \tilde{k}_0^\theta) \big(X (S_\theta^*)^2\big) \\
			&= k_0^\theta \otimes \big((S_\theta^2) X^* k_0^\theta\big) - \tilde{k}_0^\theta \otimes \big((S_\theta^2) X^* \tilde{k}_0^\theta\big).
		\end{align*}
		This completes the proof.
	\end{proof}

	\begin{proposition}\label{prop:mixed-second-order}
		Let
		\[
		T=A_\vp+B_\psi,
		\]
		where $A_\vp$ is a truncated Toeplitz operator and $B_\psi$ is a truncated Hankel operator on $\clk_\theta$, with symbols $\vp,\psi\in L^\infty(\T)$. Then $D_HD_T(T)$ is a finite-rank operator of
		rank at most $4$.
		
		More precisely, suppose
		\[
		A_\vp-S_\theta A_\vp S_\theta^*
		=
		u_1\otimes k_0^\theta+k_0^\theta\otimes v_1
		\]
		and
		\[
		B_\psi-S_\theta^*B_\psi S_\theta^*
		=
		u_2\otimes k_0^\theta+\widetilde k_0^\theta\otimes v_2 .
		\]
		Then
		\begin{align*}
			D_HD_T(T)
			&= k_0^\theta \otimes
			\big[v_1+ \overline{\theta(0)} S_\theta v_2+S_\theta^2B_\psi^*k_0^\theta\big] \\
			&\quad +\widetilde k_0^\theta\otimes
			\big[{\theta(0)}S_\theta v_1+v_2
			-S_\theta^2B_\psi^*\widetilde k_0^\theta\big] \\
			&\quad +(u_1+u_2)\otimes k_0^\theta
			-\big(S_\theta^*u_1+S_\theta u_2\big)\otimes S_\theta k_0^\theta .
		\end{align*}
	\end{proposition}
	
	\begin{proof}
		By linearity,
		\[
		D_HD_T(T)=D_HD_T(A_\vp)+D_HD_T(B_\psi).
		\]
		For the truncated Toeplitz part, Sarason's displacement identity gives
		\[
		D_T(A_\vp)=u_1\otimes k_0^\theta+k_0^\theta\otimes v_1.
		\]
		For the truncated Hankel part, we use
		\[
		D_HD_T(B_\psi)=D_TD_H(B_\psi)+[D_H,D_T](B_\psi).
		\]
		Since
		\[
		D_H(B_\psi)=u_2\otimes k_0^\theta+\widetilde k_0^\theta\otimes v_2,
		\]
		we obtain
		\[
		\begin{aligned}
			D_HD_T(T)
			&=
			D_H\bigl(u_1\otimes k_0^\theta+k_0^\theta\otimes v_1\bigr)  \\
			&\quad
			+
			D_T\bigl(u_2\otimes k_0^\theta+\widetilde k_0^\theta\otimes v_2\bigr)
			+
			[D_H,D_T](B_\psi).
		\end{aligned}
		\]
		Using
		\[
		D_T(x\otimes y)=x\otimes y-S_\theta x\otimes S_\theta y
		\]
		and
		\[
		D_H(x\otimes y)=x\otimes y-S_\theta^*x\otimes S_\theta y,
		\]
		together with Lemma~\ref{lem:commutator-displacements}, we get
		\begin{align*}
			D_HD_T(T)
			&= u_1\otimes k_0^\theta
			-(S_\theta^*u_1)\otimes(S_\theta k_0^\theta)
			+k_0^\theta\otimes v_1
			-(S_\theta^*k_0^\theta)\otimes(S_\theta v_1)\\
			&\quad
			+u_2\otimes k_0^\theta
			-(S_\theta u_2)\otimes(S_\theta k_0^\theta)
			+\widetilde k_0^\theta\otimes v_2
			-(S_\theta\widetilde k_0^\theta)\otimes(S_\theta v_2)\\
			&\quad
			+k_0^\theta\otimes \big(S_\theta^2B_\psi^*k_0^\theta\big)
			-\widetilde k_0^\theta\otimes
			\big(S_\theta^2B_\psi^*\widetilde k_0^\theta\big).
		\end{align*}
		Now use the identities
		\[
		S_\theta^*k_0^\theta
		=
		-\overline{\theta(0)}\,\widetilde k_0^\theta,
		\qquad
		S_\theta\widetilde k_0^\theta
		=
		-\theta(0)k_0^\theta .
		\]
		
		We obtain
		\begin{align*}
			D_HD_T(T)
			&= k_0^\theta \otimes
			\big[v_1+\overline{\theta(0)}S_\theta v_2
			+S_\theta^2B_\psi^*k_0^\theta\big] \\
			&\quad +\widetilde k_0^\theta\otimes
			\big[\theta(0)S_\theta v_1+v_2
			-S_\theta^2B_\psi^*\widetilde k_0^\theta\big] \\
			&\quad +(u_1+u_2)\otimes k_0^\theta
			-\big(S_\theta^*u_1+S_\theta u_2\big)\otimes S_\theta k_0^\theta .
		\end{align*}
		Thus $D_HD_T(T)$ is a sum of at most four rank-one operators. Hence
		\[
		\operatorname{rank} D_HD_T(T)\le 4.
		\]
		This completes the proof.
	\end{proof}

	Now, we define the following spaces
	\[
	\clr_T
	=
	\{u\otimes k_0^\theta+k_0^\theta\otimes v:
	u,v\in \clk_\theta\},
	\]
	and 
	\[
	\clr_H
	=
	\{x\otimes k_0^\theta+\widetilde k_0^\theta\otimes y:
	x,y\in \clk_\theta\},
	\]
	then, by \eqref{eq:sarason}, $A$ is a TTO if and only if $D_T(A)\in\mathcal{R}_T$; and, by \eqref{eq:guma}, $B$ is THO if and only if $D_H(B)\in\mathcal{R}_H$.

	\begin{theorem}\label{thm:TTO-THO-characterization}
		Let $T\in\clb(\clk_\theta)$. Then $T$ can be written as
		\[
		T=A+B,
		\]
		where $A$ is a truncated Toeplitz operator and $B$ is a truncated Hankel operator on $\clk_\theta$, if and only if there exist
		\[
		R\in\clr_T
		\qquad\text{and}\qquad
		B\in\clb(\clk_\theta)
		\]
		such that
		\[
		D_H(B)\in\clr_H,
		\quad 
		D_HD_T(T)=D_H(R)+D_HD_T(B).
		\]
	\end{theorem}
	
	\begin{proof}
		Suppose first that
		\[
		T=A+B,
		\]
		where $A$ is a TTO  and $B$ is a THO. By Sarason's characterization \eqref{eq:sarason},
		\[
		D_T(A)\in\clr_T.
		\]
		Put
		\[
		R:=D_T(A)\in\clr_T.
		\]
		Since $B$ is a THO, the Gu--Ma characterization \eqref{eq:guma} gives
		\[
		D_H(B)\in\clr_H.
		\]
		By linearity,
		\[
		D_HD_T(T)
		=
		D_HD_T(A)+D_HD_T(B)
		=
		D_H(R)+D_HD_T(B).
		\]
		
		Conversely, suppose there exist $R\in\clr_T$ and $B\in\clb(\clk_\theta)$ such that
		\[
		D_H(B)\in\clr_H
		\]
		and
		\[
		D_HD_T(T)=D_H(R)+D_HD_T(B).
		\]
		Then, by \eqref{eq:guma}, $B$ is a THO. Moreover,
		\[
		D_H\bigl(D_T(T)-R-D_T(B)\bigr)=0.
		\]
		By Lemma~\ref{lem:inj}, $D_H$ is injective in $\clb(\clk_\theta)$. Therefore
		\[
		D_T(T)-R-D_T(B)=0.
		\]
		Thus
		\[
		D_T(T-B)=R\in\clr_T.
		\]
		By Sarason's characterization \eqref{eq:sarason}, $T-B$ is a TTO. Hence
		\[
		T=(T-B)+B
		\]
		is a sum of a TTO and a THO.
	\end{proof}

	The decomposition in Theorem \ref{thm:TTO-THO-characterization} is in general not unique. By \eqref{eq:sarason} and \eqref{eq:guma}, we have
	\[
	\cli_\theta = \{C \in \clb(\clk_\theta) : D_T(C) \in \clr_T \text{ and } D_H(C) \in \clr_H\}.
	\]
	Theorem \ref{thm:TTO-THO-characterization} says exactly that $T \in \cls_\theta$ if and only if the displayed displacement equation is solvable for some $R \in \clr_T$ and some $B \in \clb(\clk_\theta)$ with $D_H(B) \in \clr_H$ and $D_HD_T(T)=D_H(R)+D_HD_T(B)$. 
	
	\begin{lemma}\label{lem:uniqueness}
		Let
		\[
		T=A+B=A'+B'
		\]
		be two decompositions with $A,A'\in\clt_\theta$ and $B,B'\in\clh_\theta$. Then
		\[
		A-A'=B'-B\in\cli_\theta,
		\]
		where
		\[
		\cli_\theta=\clt_\theta\cap\clh_\theta.
		\]
		Conversely, if $C\in\cli_\theta$ and $T=A+B$ is such a decomposition, then
		\[
		T=(A+C)+(B-C)
		\]
		is again a decomposition of $T$ as a sum of a truncated Toeplitz operator and a truncated Hankel operator. Hence the decomposition is unique modulo $\cli_\theta$.
	\end{lemma}
	
	\begin{proof}
		Suppose
		\[
		T=A+B=A'+B',
		\]
		where $A,A'\in\clt_\theta$ and $B,B'\in\clh_\theta$. Set
		\[
		C:=A-A'.
		\]
		Then
		\[
		C=B'-B.
		\]
		Since $\clt_\theta$ and $\clh_\theta$ are linear subspaces, $C$ is both a TTO and a THO. Hence
		\[
		C\in\clt_\theta\cap\clh_\theta=\cli_\theta.
		\]
		
		Conversely, let $C\in\cli_\theta$ and suppose $T=A+B$, where $A\in\clt_\theta$ and $B\in\clh_\theta$. Since $C\in\cli_\theta$, we have $C\in\clt_\theta$ and $C\in\clh_\theta$. Therefore,
		\[
		A+C\in\clt_\theta
		\qquad\text{and}\qquad
		B-C\in\clh_\theta.
		\]
		Moreover,
		\[
		(A+C)+(B-C)=A+B=T.
		\]
		Thus every decomposition is determined only up to addition and subtraction of an element of $\cli_\theta$.
	\end{proof}
	
	\begin{remark}
		At the symbol level, there is an additional ambiguity coming from the kernels of the symbol maps. Suppose
		\[
		A_\vp+B_\psi=A_{\vp'}+B_{\psi'}.
		\]
		By Lemma~\ref{lem:uniqueness}, there exists $C\in\cli_\theta$ such
		that
		\[
		C=A_\vp-A_{\vp'}=B_{\psi'}-B_\psi.
		\]
		Choose symbols $\vp_C\in L^\infty(\T)$ and
		$\psi_C\in L^\infty(\T)\cap\overline{zH^2}$ such that
		\[
		A_{\vp_C}=C
		\qquad\text{and}\qquad
		B_{\psi_C}=C.
		\]
		Then $A_{\varphi}=A_{\varphi'+\varphi_{C}}$ and
		$B_{\psi'}=B_{\psi+\psi_{C}}$, and hence
		\[
		\varphi-(\varphi'+\varphi_{C})\in\theta H^{2}+\overline{\theta H^{2}}
		\]
		by the kernel description of the TTO symbol map, while
		\[
		\psi'-(\psi+\psi_{C})\ \text{lies in the kernel of the THO symbol map }\ \psi\mapsto B_{\psi}.
		\]
		This kernel is nontrivial, so a truncated Hankel operator does not determine its symbol uniquely. Thus the non-uniqueness of symbols comes from both the operator-level ambiguity $\cli_\theta$ and the kernels of the two symbol maps.
	\end{remark}
	
	In particular, the ambiguity space $\cli_\theta$ measures the failure of uniqueness of the decomposition. We now identify this space explicitly in the finite-dimensional case $\theta(z)=z^n$.
	
	\begin{remark}\label{rem:Izn}
		Let $\theta(z)=z^n$, $n\ge 2$. With respect to the basis
		\[
		e_j=z^j,\qquad 0\le j\le n-1,
		\]
		of $\clk_\theta$, TTOs are precisely the $n\times n$ Toeplitz matrices, while THOs are precisely the $n\times n$ Hankel matrices. Hence
		\[
		\cli_{z^n}
		=
		\{M\in M_n(\C): M \text{ is both Toeplitz and Hankel}\}.
		\]
		
		Let $M=(M_{ij})_{i,j=0}^{n-1}\in\cli_{z^n}$. Since $M$ is Toeplitz and Hankel, there exist scalar sequences $(\alpha_r)$ and $(\beta_s)$ such that
		\[
		M_{ij}=\alpha_{i-j}=\beta_{i+j},
		\qquad 0\le i,j\le n-1.
		\]
		Thus the entry $M_{ij}$ depends simultaneously on $i-j$ and on $i+j$. Since $i-j$ and $i+j$ have the same parity, all entries with $i+j$ even have one common value, and all entries with $i+j$ odd have another common value. Therefore $\cli_{z^n}$ is spanned by the two checkerboard matrices
		\[
		(E_{\mathrm{even}})_{ij}
		=
		\begin{cases}
			1, & i+j \text{ even},\\
			0, & i+j \text{ odd},
		\end{cases}
		\qquad
		(E_{\mathrm{odd}})_{ij}
		=
		\begin{cases}
			0, & i+j \text{ even},\\
			1, & i+j \text{ odd}.
		\end{cases}
		\]
		Both matrices are Toeplitz and Hankel, and they are linearly independent. Hence
		\[
		\dim \cli_{z^n}=2,
		\qquad n\ge 2.
		\]
	\end{remark}

	\begin{corollary}\label{cor:zn}
		Let $\theta(z)=z^n$ for $n\ge 2$. Then
		\[
		\clk_\theta=\operatorname{span}\{1,z,\ldots,z^{n-1}\}.
		\]
		With respect to the basis $e_j=z^j$, $0\le j\le n-1$, an operator $T=(T_{ij})_{i,j=0}^{n-1}$ belongs to $\cls_\theta$ if and only if
		\[
		T_{i-1,j}+T_{i+1,j}=T_{i,j-1}+T_{i,j+1},
		\qquad 1\le i,j\le n-2.
		\]
		Equivalently, $T$ is a Toeplitz-plus-Hankel matrix in the classical sense of Bevilacqua--Bonanni--Bozzo \cite{BBB}.
	\end{corollary}
	
	\begin{proof}
		For $\theta(z)=z^n$, we have
		\[
		k_0^\theta=1=e_0,
		\qquad
		\widetilde k_0^\theta=z^{n-1}=e_{n-1}.
		\]
		Moreover,
		\[
		S_\theta e_j=e_{j+1}\quad (0\le j<n-1),
		\qquad
		S_\theta e_{n-1}=0,
		\]
		and
		\[
		S_\theta^*e_j=e_{j-1}\quad (1\le j\le n-1),
		\qquad
		S_\theta^*e_0=0.
		\]
		Thus, for a matrix $X=(X_{ij})$, with the convention that out-of-range indices give zero, we have
		\begin{equation}\label{eq:DTDH-matrix}
			(D_T X)_{ij}=X_{ij}-X_{i-1,j-1},
			\qquad
			(D_H X)_{ij}=X_{ij}-X_{i+1,j-1}.
		\end{equation}
		
		Since $k_0^\theta=e_0$, the defect space
		\[
		\mathcal R_T
		=
		\{u\otimes k_0^\theta+k_0^\theta\otimes v:u,v\in\clk_\theta\}
		\]
		consists of matrices supported only in the first column and the first row. Hence, by Sarason's characterization \eqref{eq:sarason}, and by \eqref{eq:DTDH-matrix}, an operator $A$ is a truncated Toeplitz operator if and only if
		\[
		A_{ij}=A_{i-1,j-1},
		\qquad 1\le i,j\le n-1.
		\]
		Thus $A_{ij}$ depends only on the diagonal index $i-j$, and so $A$ is a Toeplitz matrix.
		
		Similarly, since $k_0^\theta=e_0$ and $\widetilde k_0^\theta=e_{n-1}$, the defect space
		\[
		\mathcal R_H
		=
		\{x\otimes k_0^\theta+\widetilde k_0^\theta\otimes y:x,y\in\clk_\theta\}
		\]
		consists of matrices supported only in the first column and the last row. Hence, by the Gu--Ma characterization \eqref{eq:guma}, and by \eqref{eq:DTDH-matrix}, an operator $B$ is a truncated Hankel operator if and only if
		\[
		B_{ij}=B_{i+1,j-1},
		\qquad 0\le i\le n-2,\ 1\le j\le n-1.
		\]
		Thus $B_{ij}$ depends only on the antidiagonal index $i+j$, and so $B$ is a Hankel matrix.
		
		Consequently, $\mathcal S_\theta=\mathcal T_\theta+\mathcal H_\theta$ is precisely the space of $n\times n$ Toeplitz-plus-Hankel matrices.
		
		It remains to connect this with the four-point recurrence. If $A$ is Toeplitz, then $A_{ij}=\alpha_{i-j}$, and therefore
		\[
		A_{i-1,j}+A_{i+1,j}
		=
		\alpha_{i-1-j}+\alpha_{i+1-j}
		=
		A_{i,j-1}+A_{i,j+1}.
		\]
		Similarly, every Hankel matrix satisfies the same recurrence. By linearity, every Toeplitz-plus-Hankel matrix satisfies
		\[
		T_{i-1,j}+T_{i+1,j}=T_{i,j-1}+T_{i,j+1},
		\qquad 1\le i,j\le n-2.
		\]
		Conversely, the Bevilacqua--Bonanni--Bozzo theorem says that every matrix satisfying this recurrence is Toeplitz-plus-Hankel \cite{BBB}. Hence $T\in\cls_\theta$ if and only if the stated recurrence holds.
	\end{proof}
	
	\begin{remark}\label{rem:T+H-matrix-form}
		The matrix form of an operator $T\in\mathcal S_\theta$ for $\theta(z)=z^n$ is transparent from Corollary~\ref{cor:zn}. Writing $T=A+B$, with $A$ Toeplitz and $B$ Hankel, we have
		\[
		T_{ij}=\alpha_{i-j}+\beta_{i+j},
		\qquad 0\le i,j\le n-1,
		\]
		where $\alpha_{-(n-1)},\dots,\alpha_{n-1}$ are the diagonal parameters of $A$, and $\beta_0,\dots,\beta_{2n-2}$ are the antidiagonal parameters of $B$. Explicitly,
		\[
		T=\begin{pmatrix}
			\alpha_0+\beta_0          & \alpha_{-1}+\beta_1      & \alpha_{-2}+\beta_2      & \cdots & \alpha_{-(n-1)}+\beta_{n-1}   \\
			\alpha_1+\beta_1          & \alpha_0+\beta_2         & \alpha_{-1}+\beta_3      & \cdots & \alpha_{-(n-2)}+\beta_n       \\
			\alpha_2+\beta_2          & \alpha_1+\beta_3         & \alpha_0+\beta_4         & \cdots & \alpha_{-(n-3)}+\beta_{n+1}   \\
			\vdots                    & \vdots                   & \vdots                   & \ddots & \vdots                        \\
			\alpha_{n-1}+\beta_{n-1}  & \alpha_{n-2}+\beta_n     & \alpha_{n-3}+\beta_{n+1} & \cdots & \alpha_0+\beta_{2n-2}
		\end{pmatrix}.
		\]
		This form immediately gives the recurrence
		\[
		T_{i-1,j}+T_{i+1,j}=T_{i,j-1}+T_{i,j+1},
		\qquad 1\le i,j\le n-2.
		\]
	\end{remark}

	\section*{Declarations}
	\noindent\textbf{Funding.} Not applicable. \\
	\noindent\textbf{Conflict of interest.} The authors declare that they have no conflicts of interest. \\
	\noindent\textbf{Data availability.} This work is purely theoretical, and no datasets are used here. 

\end{document}